\documentclass[11pt]{amsart}
\usepackage{fullpage}

\usepackage{amssymb}

\usepackage[english]{babel}
\theoremstyle{plain}
\newtheorem{theorem}                {Theorem}      [section]

\newtheorem{lemma}        [theorem]  {Lemma}

\theoremstyle{definition}

\newtheorem{remark}       [theorem]  {Remark}
\newtheorem{definition}   [theorem]  {Definition}

\newtheorem{property}      [theorem]  {The Almansi property}
\newtheorem{property-general-semi}      [theorem]  {The generalized Almansi property}
\newtheorem{property-general}      [theorem]  {The weak Almansi property}

\DeclareMathOperator{\Div}{div}

\numberwithin{equation}{section}

\def \R{{\mathbb R}}
\def \s{{\mathbb S}}
\def \z{{\mathbb Z}}
\def \h{{\mathbb H}}
\def \1{~}

\numberwithin{equation}{section}

\begin{document}

\title[Almansi property]{A note on the Almansi property}

\author{S.~Montaldo}
\address{Universit\`a degli Studi di Cagliari\\
Dipartimento di Matematica e Informatica\\
Via Ospedale 72\\
09124 Cagliari, Italia}
\email{montaldo@unica.it}

\author{A.~Ratto}
\address{Universit\`a degli Studi di Cagliari\\
Dipartimento di Matematica e Informatica\\
Via Ospedale 72\\
09124 Cagliari, Italia}
\email{rattoa@unica.it}

\begin{abstract}
The first goal of this note is to study the Almansi property on an $m$-dimensional model in the sense of Greene and Wu and, more generally, in a Riemannian geometric setting. In particular, we shall prove that the only model on which the Almansi property is verified is the Euclidean space $\R^m$. In the second part of the paper we shall study Almansi's property and biharmonicity for functions which depend on the distance from a given submanifold.
Finally, in the last section we provide an extension to the semi-Euclidean case $\R^{p,q}$ which includes the proof of the classical Almansi property in $\R^m$ as a special instance.

\end{abstract}

\thanks{Corresponding author e-mail: montaldo@unica.it\\
Work supported by P.R.I.N. 2015 -- Variet\`a reali e complesse: geometria, topologia e analisi armonica -- Italy, and G.N.S.A.G.A., INdAM, Italy.}



\maketitle
\begin{flushright}
\textit{Dedicated to Prof. Renzo Caddeo}
\end{flushright}
\section{Introduction}\label{intro}
Let $(M^m,g)$ be a smooth, $m$-dimensional Riemannian manifold with metric $g$. If $F$ denotes a smooth, real valued function on $M^m$, the well-known {\it Laplace-Beltrami} operator $\Delta$ can be described with respect to local coordinates by
\begin{equation}\label{Laplacian}
\Delta F=\Div (\nabla F)=\frac{1}{\sqrt{|g|}}\, \frac{\partial}{\partial x_j}\left(g^{ij}\, \sqrt{|g|}\, \frac{\partial F}{\partial x_i}\right).
\end{equation}
\begin{definition}\label{iterated-laplacian}
For a positive integer $s$ the iterated Laplace-Beltrami operator $\Delta^s$ is defined by
$$
\Delta^{0} F=F,\quad \Delta^{s} F=\Delta(\Delta^{(s-1)}F).
$$
\end{definition}
\begin{definition}\label{definition-proper-r-harmonic} For a positive integer $s$ we say that a smooth, real-valued function $F:(M^m,g)\to\R$ is
\begin{enumerate}
\item[(a)] {\it $s$-harmonic} if $\Delta^s F=0$,
\item[(b)] {\it proper $s$-harmonic} if $\Delta^s F=0$ and $\Delta^{(s-1)} F$ does not vanish identically.
\end{enumerate}
\end{definition}

In particular, it should be noted that {\it harmonic} functions are precisely the $1$-harmonic ones, while {\it biharmonic} functions coincide with the $2$-harmonic.  In some texts, $s$-harmonic functions are also called {\it polyharmonic} of order $s$ (see \cite{ACL} for background and related analytical topics).

The classical Almansi property can be stated as follows (see \cite{Almansi, Caddeo2}).

\begin{property}\label{Almansi} Let $\R^m$ be equipped with its canonical Euclidean metric. Let $H:\R^m \to \R$ be defined by
\begin{equation}\label{definizione-H}
H(x_1,\ldots,x_m)=c_1\,\sum_{j=1}^m \, x_j^2 +c_2\,\, ,
\end{equation}
where $c_1,\,c_2$ are two real constants with $c_1 \neq 0$. If $F:U\subset\R^m \to \R$ is an $s$-harmonic function ($s \geq1$) on an open set $U$, then the product function $G=H \, F$ is $(s+1)$-harmonic on $U$. Moreover, if $F$ is a proper $s$-harmonic on the whole $\R^m$, then $G$ is proper $(s+1)$-harmonic on $\R^m$.
\end{property}

As a consequence of the Almansi property, any biharmonic function can be represented in terms of two harmonic functions and, more generally, any $r$-harmonic function can be described by means of $r$ harmonic functions (see Proposition~1.3 of \cite{ACL}). In the same spirit, in Chapter~10 of  \cite{dassios}, ellipsoidal biharmonic functions on $\R^3$ were described by means of their relation to ellipsoidal harmonics via the Almansi property. Moreover, as shown in \cite{dassios2}, the Almansi decomposition is related to the Kelvin transformation and provides a way to solve analytical problems in potential theory and Stokes flow which could be difficult to solve by the classical spectral method. 

In {\cite{GMR}} we started a program of investigation of $r$-harmonic functions in a Riemannian geometric setting and, in particular, we provided several new families of examples on open subsets of the classical compact simple Lie groups. The main aim of this paper is to investigate up to which extent Almansi-type methods can be used to produce new examples on Riemannian manifolds which are different from the Euclidean space.
More specifically, in a Riemannian setting the function $H$ in \eqref{definizione-H} can be conveniently rewritten as
\begin{equation}\label{definizione-H-bis}
H=H(r)=c_1\,r^2 +c_2\,\, ,
\end{equation}
where $r$ denotes the distance function from a fixed point. Now, we observe that, since the constant function $F \equiv 1$ is trivially proper harmonic on any Riemannian manifold, the validity of the Almansi property in a Riemannian context requires that the function $H$ be proper biharmonic. These facts have led Caddeo (see \cite{Caddeo1, Caddeo2}) to investigate the biharmonicity of the function $H(r)=r^k$ ($k \in \z$) on a general Riemannian manifold. This study was based on the local description of $\Delta$ with respect to a set of normal coordinates as described in \cite{VanHecke}: these coordinates offer great generality, but also imply high technical difficulties. These considerations have suggested to us to restrict our attention to a specific class of manifolds (which include all the space forms) on which we can carry out a study of the biharmonicity of any function of the type $H=H(r)$. More precisely, the first aim of this note is to study a version of the Almansi property for biharmonic functions on \emph{models} in the sense of \cite{GW} (see also \cite{Petersen}). In the recent past, this type of manifolds have provided a suitable setting to construct several new examples of biharmonic maps (see \cite{MOR, MR}, and \cite{MO} for a survey on this topic). We recall that an $m$-dimensional manifold $(M^m(o),\,g)$ with a pole $o$ is a model if and only if every linear isometry of $T_oM$ can be realized as the differential at $o$ of an isometry of $M$. A significant geometric property of a model is the fact that we can describe it by means of geodesic polar coordinates centered at the pole $o$, as follows:
\begin{equation}\label{model}
(M^m(o),\,g) = \left (\,\s^{m-1} \times [0,\, + \infty) , \, f^2(r)\, g_{\s^{m-1}} \, + \, dr^2 \,\right ) \,\, ,
\end{equation}
where
$$ (\,\s^{m-1},\, g_{\s^{m-1}} \, )
$$
denotes the $(m-1)$-dimensional Euclidean unit sphere, and the function $f(r)$ is a smooth function which satisfies
\begin{equation}\label{condizioni-su-f}
\left \{
  \begin{array}{l}
    f(0)=0 \,, \quad f'(0)=1 \quad {\rm and}\quad f(r)>0 \quad {\rm if} \,\, r>0 \,\, ; \\
    \,\\
    f^{(2k)}(0)=0 \quad {\rm for} \,\, {\rm all } \,\, k \geq 1 \,\, .\\
  \end{array}
\right .
\end{equation}
We also note that $r$ measures the geodesic distance from the pole $o$. To shorten notation and emphasize the role of the function $f$, we shall write $M_f^m(o)$ to denote a model as in \eqref{model}.

\begin{remark}\label{remark-su-R^m} We observe that, if $f(r)=r$, then $M_f^m(o)=\R^m$. In particular, the function $H$ given in \eqref{definizione-H-bis} is a special instance of a function on $M_f^m(o)$ which just depends on the distance from the pole. For future reference it is useful to remark that the function $r^2$ is smooth across the pole of $M_f^m(o)$, while $H(r)=r$ is not. More precisely, we point out that a $\mathcal{C}^\infty$ radial function $H=H(r)$ ($r \geq0$) gives rise to a function on $M_f^m(o)$ which is smooth across the pole if and only if (see \cite{Petersen}) its derivatives at $r=0$ satisfy
\begin{equation}\label{smoothness-across-pole}
H^{(2k-1)}(0)=0 \quad {\rm for} \,\, {\rm all } \,\, k \geq 1 \,\, .
\end{equation}
\end{remark}
\begin{remark}\label{remark-su-altri-modelli}
If $f(r)= \sinh r$, then $M_f^m(o)$ is isometric to the $m$-dimensional hyperbolic space $\h^m$. With a slight abuse of terminology, we shall also admit the case that $f(r)$ is defined on a finite interval $[0,\,b]$, with $f(b)=0$, $f'(b)=-1$ and  $f^{(2k)}(b)=0$  {for} all  $k \geq 1$. In particular, if $f(r)=\sin r $ and $\, 0 \leq r \leq \pi $, our manifold becomes the Euclidean unit $m$-sphere $\s^m$.
\end{remark}

For the sake of completeness, we also recall that the radial curvature $K(r)$ ($r>0$) of a model $M_f^m(o)$ is defined as the sectional curvature of any plane which contains $\partial \slash\, \partial  r$. The radial curvature is related to the function $f(r)$ by means of the following fundamental equation (the Jacobi equation, see \cite{GW}):
\begin{equation}\label{equazionecurvaturaradiale}
    f''(r)\,+\,K(r)\,f(r)\,=\,0 \,\, , \quad r>0 \,\, .
\end{equation}
Because of the previous discussion and Remark \ref{remark-su-R^m}, a first natural step is to study the following weak version of the Almansi property \ref{Almansi}:
\begin{property-general}\label{Almansi-model} Let $M_f^m(o)$ be an $m$-dimensional model as in \eqref{model}. We say that the weak Almansi property holds on $M_f^m(o)$ if there exists a smooth, radial function $H=H(r)$ on $M_f^m(o)$ such that: if $F:U \subset M_f^m(o) \to \R$ is a locally defined harmonic function, then the product function $G=H \, F$ is biharmonic on the open set $U$; and if $F$ is proper harmonic on the whole $M_f^m(o)$, then $G=H \, F$ is proper biharmonic on $M_f^m(o)$.
\end{property-general}
We shall obtain the following result:
\begin{theorem}\label{main-theorem} Let $M_f^m(o)$ be an $m$-dimensional model on which the weak Almansi property \ref{Almansi-model} holds. Then $M_f^m(o)=\R^m$ and the function $H$ coincides with \eqref{definizione-H}.
\end{theorem}
In the direction of positive answers, in Section\1\ref{section-distanza-da-varieta} we shall introduce a suitable class of warped products and prove a version of the Almansi property for functions which depend on the distance from a given submanifold (see Theorem\1\ref{theorem-2-warped} below). Also, in Section\1\ref{almansi-semi} we provide an extension of the classical Almansi property to the semi-Euclidean case.

\noindent{\bf Acknowledgements}. The authors wish to thank the referee for several useful comments and suggestions which improved the quality of the paper.

\section{Proof of Theorem \ref{main-theorem}}

First, we observe that, as an easy consequence of the definitions \eqref{Laplacian} and \eqref{model}, the Laplacian of a radial function $F=F(r)$ on a model $M_f^m(o)$ is given by:
\begin{equation}\label{Laplacian-radiale}
    \Delta F(r)= \frac {1}{f^{m-1}(r)}\, \left( f^{m-1}(r)\,F'(r) \right)'=F''(r) + (m-1)\, \frac{f'(r)}{f(r)}\, F'(r) \,\,,
\end{equation}
where $'$ denotes derivative with respect to $r$. We also recall here the general formula for the Laplacian of a product of two functions:
\begin{equation}\label{Laplaciano-prodotto}
\Delta (F_1 \,F_2)= F_1 \, \Delta F_2+ F_2 \, \Delta F_1 +2 \,\nabla F_1 \cdot \nabla F_2.
\end{equation}
In particular, we point out that, if $F_1,\,F_2$ are two radial functions on a model, the product formula \eqref{Laplaciano-prodotto} takes the simple form:
\begin{equation}\label{Laplaciano-prodotto-radiale}
\Delta (F_1 \,F_2)= F_1 \, \Delta F_2+ F_2 \, \Delta F_1 +2 \,F_1'\,F_2'.
\end{equation}
Now we provide the proof of Theorem \ref{main-theorem}.
\begin{proof} First, we note that according to \eqref{Laplacian-radiale} a radial function on a model $M_f^m(o)$ is harmonic if and only if
\begin{equation}\label{condizione-harmonicity}
f^{m-1}(r) \, F'(r)= c_F ,
\end{equation}
where $c_F$ is a real constant. To simplify the notation we shall write, for any radial function $F$,
\begin{equation}\label{tau-radiale}
    \tau_F (r)= F''(r) + (m-1)\, \frac{f'(r)}{f(r)}\, F'(r) \,\,.
\end{equation}
Now, let us assume that a pair $M_f^m(o)$, $H(r)$ verify the weak Almansi property \ref{Almansi-model}. Let $F$ be a locally defined proper radial harmonic function, so that \eqref{condizione-harmonicity} holds (note that the set of such functions is always not empty, because $F\equiv 1$ is proper harmonic). Now, we set $G(r)=F(r) \, H(r)$ and, by using \eqref{Laplaciano-prodotto-radiale}, \eqref{tau-radiale} and the harmonicity of $F$, we deduce:
\begin{equation}\label{step1}
\left \{
\begin{array}{ll}
{\rm (i)} &\tau_G (r)=\tau_H (r)\,F (r)+2\, H'(r)\,F'(r) \\
& \\
{\rm (ii)}& \tau_G ' (r)=\tau_H '(r)\,F (r)+\,F' (r)\,\displaystyle{\left [\tau_H (r)+2\,H''(r)-2\,(m-1)\, H'(r) \frac{f'(r)}{f(r)}\right ]}
\end{array}
\right .
\end{equation}
As we explained in the introduction, we know that $H$ must be proper biharmonic. Arguing as in the case of \eqref{condizione-harmonicity}, we deduce that
\begin{equation}\label{step2}
\left \{
\begin{array}{ll}
{\rm (i)} &\tau_H (r) \not \equiv 0\\
& \\
{\rm (ii)} &f^{m-1}(r)\,\tau_H ' (r)=c_H\,\,,
\end{array}
\right .
\end{equation}
where $c_H$ is a real constant. Now, since the function $f$ verifies the conditions in \eqref{condizioni-su-f} and $H$ is smooth across the pole, the only acceptable instance in \eqref{step2} is $c_H=0$, from which it follows easily that
\begin{equation}\label{condizione-eplicita-H}
\tau_H (r)=d_H  \,\, ,
\end{equation}
where $d_H$ is a nonzero constant.
Next, we analyse the consequences of \eqref{condizione-eplicita-H} on \eqref{step1}. Indeed, let us first note that
\begin{equation}\label{condizione-biharmonicity-G}
\tau_G ' (r)=\frac{c_G}{f^{m-1}(r)}
\end{equation}
for some real constant $c_G$ because $G$ is biharmonic. Then, by using \eqref{condizione-eplicita-H}, \eqref{condizione-biharmonicity-G} and \eqref{condizione-harmonicity} into \eqref{step1}(ii) we obtain:
\begin{equation}\label{step4}
\frac{c_G}{f^{m-1}(r)}=\frac{c_F}{f^{m-1}(r)} \left[d_H+2 \,\left(d_H- \,(m-1)\, H'(r) \frac{f'(r)}{f(r)}\right ) -2\,(m-1)\, H'(r) \frac{f'(r)}{f(r)}\right ] \,\,.
\end{equation}
Next, we easily deduce from \eqref{step4} that
$$
H'(r) \frac{f'(r)}{f(r)}=C
$$
for some $C\neq0$, so that it follows immediately from \eqref{condizione-eplicita-H} that also $H''(r)$ must be a constant. Now \eqref{smoothness-across-pole} immediately implies that $H(r)$ is of the type \eqref{definizione-H-bis}. Finally, integrating
$$
\frac{f'(r)}{f(r)}=\frac{C^*}{r}
$$
($C^* \neq 0$) with boundary conditions $f(0)=0,\,f'(0)=1$ we conclude that the only acceptable case is $C^*=1$ and $f(r)=r$, so that $M_f^m(o)=\R^m$ and the proof of Theorem~\ref{main-theorem} is completed.
\end{proof}

\section{Extensions to the case that $r$ measures the distance from a $k$-dimensional submanifold ($k \geq1$)}\label{section-distanza-da-varieta}
As we mentioned in the introduction, the biharmonicity of the distance function $r$ was studied by Caddeo (\cite{Caddeo1}). In particular, he restricted his attention to Euclidean spaces and rank $1$ symmetric spaces and found that, away from the locus $r=0$, the equation $\Delta^2 r=0$ holds on the real line and on the $3$-dimensional space forms. The main aim of this section is to study the biharmonicity of $r$ in a context where $r$ represents the distance not from a fixed point, but from a given submanifold. In particular, in Lemma\1\ref{theorem-1-warped} below we shall obtain an affirmative answer under a specific dimensional restriction: in this favorable case, we shall prove in Theorem\1\ref{theorem-2-warped} a version of the Almansi property for functions which depend on $r$. Let us start the description of our geometric setting by observing that the Riemannian manifold
\begin{equation}\label{join}
(M,g)=\left (\s^{p-1} \times \left [0, \frac{\pi}{2}\right ) \times \s^{q-1}, \sin^2 r \,g_{_{\s^{p-1}}}+dr^2+\cos^2 r \,g_{_{\s^{q-1}}} \right )
\end{equation}
is isometric to $\s^{p+q-1} \setminus \s^{p-1}$, i.e., the Euclidean $(p+q-1)$-sphere minus a focal variety $\s^{p-1}$. We recall that focal varieties are the singular level sets of isoparametric functions (for a survey on isoparametric functions and their associated families of parallel hypersurfaces  we refer to \cite{thorbergsson} or, for a more recent reference, to Chapter 3 of \cite{cecil}).
The isoparametric function $f$ associated to our example \eqref{join}, considering $\s^{p+q-1}\subset \R^p\times \R^q$, is defined by $f=F_{|\s^{p+q-1}}$ where $F(x,y)=|y|^2-|x|^2$, $x\in\R^p$, $y\in\R^q$ and with respect to the coordinates in  \eqref{join} the expression of $f$ is simply given by $\cos(2 r)$. In this example, $r$ measures the distance from the focal variety $\s^{q-1}$ which is the locus associated to $r=0$.

In a similar fashion, the Riemannian manifold
\begin{equation}\label{join-hyperbolic}
(M,g)=\left (\s^{p-1} \times [0, +\infty) \times \s^{q-1}, \sinh^2 r \,g_{_{\s^{p-1}}}+dr^2+\cosh^2 r \,g_{_{\s^{q-1}}} \right )
\end{equation}
is isometric to a warped product of the type $\h^p \times \s^{q-1}$ and again $r$ measures the distance from the focal variety $\s^{q-1}$ (note that the Riemannian manifold \eqref{join-hyperbolic} does not have constant sectional curvature). In order to provide a unified treatment for calculations which involve both the cases \eqref{join} and \eqref{join-hyperbolic}, we consider now the following more general class of warped products:
\begin{equation}\label{general-warped-product}
\left ( M_{f_1}^{p}(o)\times \s^{q-1}, f_1^2 (r) \,g_{_{\s^{p-1}}}+dr^2+ f_2^2 (r) \,g_{_{\s^{q-1}}} \right ) \,\, ,
\end{equation}
where $M_{f_1}^{p}(o)$ is a $p$-dimensional model as illustrated in Section~\ref{intro}, and $f_2(r)$ is a smooth positive function which verify \eqref{smoothness-across-pole} (this assumption guarantees that the Riemannian metric in \eqref{general-warped-product} is smooth across the focal variety $r=0$). To simplify notation, we shall write $M_{f_1,f_2}^{p+q-1}$ to denote a warped product as in \eqref{general-warped-product}. In particular, we observe that if $f_1(r)=\sin r$ and $f_2(r)=\cos r$, then we recover \eqref{join}, while if $f_1(r)=\sinh r$ and $f_2(r)=\cosh r$ we obtain \eqref{join-hyperbolic}. Also, we note that if $f_1(r)= r$ and $f_2(r)\equiv 1$, then $M_{f_1,f_2}^{p+q-1}=\R^p \times \s^{q-1}$.

Now, let us suppose that $F=F(r)$ is a function on $M_{f_1,f_2}^{p+q-1}$ which depends just on $r$. It follows easily from \eqref{Laplacian} that
\begin{equation}\label{Laplacian-radiale-warped}
    \Delta F(r)= \frac {1}{f_1^{p-1}(r)f_2^{q-1}(r)} \left( f_1^{p-1}(r)f_2^{q-1}(r)\,F'(r) \right)'=F''(r) + \left [(p-1) \frac{f_1'(r)}{f_1(r)} +(q-1) \frac{f_2'(r)}{f_2(r)} \right ] F'(r) \,.
\end{equation}
Next, to end the preliminaries of this section, let $F^*=F^*(r,\vartheta_{p-1},\varphi_{q-1})$ be a function on $M_{f_1,f_2}^{p+q-1}$ of the following type:
\begin{equation}\label{funzione-non-radiale}
F^*(r,\vartheta_{p-1},\varphi_{q-1})=F(r)\,V\left(\vartheta_{p-1}\right)\,W\left(\varphi_{q-1}\right ) \,\,,
\end{equation}
where $V(\vartheta_{p-1})$ (respectively, $W(\varphi_{q-1})$) is an eigenfunction of $\Delta^{\s^{p-1}}$ with eigenvalue $\lambda$ (respectively, $\Delta^{\s^{q-1}}$ with eigenvalue $\mu$) (note that in our paper the sign convention for the Laplace operator is such that its spectrum (see \cite{BGM}) has negative eigenvalues). It follows again from \eqref{Laplacian} that
\begin{equation}\label{Laplacian-non-radiale-warped}
    \Delta F^*=\left\lbrace F''(r) + \left [(p-1) \frac{f_1'(r)}{f_1(r)} +(q-1) \frac{f_2'(r)}{f_2(r)} \right ] F'(r)+ \left [\frac{\lambda}{f_1^2(r)}+ \frac{\mu}{f_2^2(r)}\right ] \,\right \rbrace \, V\left(\vartheta_{p-1}\right)\,W\left(\varphi_{q-1}\right ) \,\,.
\end{equation}
Caddeo showed in \cite{Caddeo1} that, away from the locus $r=0$, the distance function $r$ from a fixed point gives rise to a version of the Almansi property for radial functions. Namely, it was proved in \cite{Caddeo1} that if $F(r)$ is a radial proper harmonic function on a $3$-dimensional space form, then $r\,F(r)$ is proper biharmonic $(r>0)$. In Theorem\1\ref{theorem-2-warped} we will show that a result of this type holds in the setting \eqref{general-warped-product}, where $r$ measures the distance from the focal variety. First, we establish the following:
\begin{lemma}\label{theorem-1-warped}
Let $M$ be as in \eqref{join} or in \eqref{join-hyperbolic}. Then the function $H(r)=r\,(r>0) $ is proper biharmonic if and only if $p=q=3$.
\end{lemma}
\begin{proof} We provide the details of the proof in the case that $M$ is as in \eqref{join}. By using \eqref{Laplacian-radiale-warped} we find:
\begin{equation}\label{Delta-r-eDelta^2-r}
\begin{cases}
\Delta r= (p-1) \cot r - (q-1) \tan r \\
\Delta^2 r= -(p-3)(p-1) \cot r \csc^2 r + (q-3) (q-1) \tan r \sec^2 r \,\,.
\end{cases}
\end{equation}
Now the thesis follows by direct inspection of \eqref{Delta-r-eDelta^2-r}.
In the case that $M$ is as in \eqref{join-hyperbolic}, one uses \eqref{Laplacian-radiale-warped} and the calculations to end the proof are entirely similar.
\end{proof}
Now, we are in the position to prove the following radial version of the Almansi property:
\begin{theorem}\label{theorem-2-warped}
Let $M$ be as in \eqref{join} or in \eqref{join-hyperbolic} and assume $p=q=3$.
\begin{itemize}
\item[{\rm (i)}] Let $r>0$. If $F(r)$ is a radial proper harmonic function on $M$, then the function $G(r)=r\,F(r)$ is proper biharmonic.
\item[{\rm (ii)}] There exists a nonradial proper harmonic function $F^*(r,\vartheta_{p-1},\varphi_{q-1}) \,$ on $M \setminus  \s^{q-1}$ such that the function $G=r\,F^*$ is not biharmonic on $M \setminus \s^{q-1}$.
\end{itemize}
\end{theorem}
\begin{proof} We provide the proof in the case that $M$ is as in \eqref{join}.

(i) We use the assumption that $F$ is proper harmonic and compute:
\begin{equation}\label{DeltaG-1}
\Delta G(r)= \Delta r \, F(r)+r\, \Delta F +2 \, \nabla r \cdot \nabla F = \Delta r \,F(r)+2 F'(r) \,\,.
\end{equation}
Since $p=q=3$, $\Delta r =4\,\cot (2r)$ and an explicit integration shows that the only function $F(r)$ which satisfies both the equations $\Delta G = 0$ and $\Delta F=0$ is $F(r) \equiv 0$. So, we conclude that under our hypotheses $G$ is not harmonic. Next, we compute the bi-Laplacian of $G$:
\begin{eqnarray}\label{DeltaG-2}
\Delta^2 G(r)&=&\Delta \left(\Delta r \,F(r) \right)+ 2\,\Delta F'(r) \\ \nonumber
&=& 2  \left [\nabla \left( \Delta r \right )\right ] \cdot \nabla F+ 2\,\Delta F'(r) \\ \nonumber
&=&2 \,\left ( 4\,\cot (2r) \right )' F'(r)+ 2  \left [F'''(r)+4\,\cot (2r)\,F''(r) \right ]\,\,, \\ \nonumber
\end{eqnarray}
where we have used that $F$ is harmonic and the fact that $\Delta^2 r=0$ by Lemma\1\ref{theorem-1-warped}. Now, the harmonicity of $F$ implies:
\begin{equation}\label{DeltaG-3}
 F'''(r)= -\,\left (4\,\cot (2r)\,F'(r)\right )'\,\,.
\end{equation}
Finally, replacing \eqref{DeltaG-3} into the last equation of \eqref{DeltaG-2}, it is easy to conclude that $\Delta^2 G(r)=0$.

(ii) In order to construct our counterexample, we take
\begin{equation}\label{counterexample}
F^*(r,\vartheta_{2},\varphi_{2})=\overline{F}(r)\,V\left(\vartheta_{2}\right)\,W\left(\varphi_{2}\right ) \,\,,
\end{equation}
where $\overline{F}(r)$ is to be determined and $V\left(\vartheta_{2}\right),\,W\left(\varphi_{2}\right ) $ are eigenfunctions on $\s^2$ both associated to eigenvalues $\lambda=\mu=-\,2$ (this choice is admissible, because the spectrum of the Laplacian in this dimension is of the form $ -k(k+1)$, $k \geq1$ (see \cite{BGM})). By using \eqref{Laplacian-non-radiale-warped} we find that a function $F^*$ as in \eqref{counterexample} is harmonic if and only if $\overline{F}(r)\,(r>0)$ is a solution of
\begin{equation}\label{harmonicity-counterexample}
F''(r)+4\,\cot (2r)\,F'(r) - \,\frac{8}{\sin^2 (2r)}\, F(r)=0 \,\, .
\end{equation}
Direct integration shows that
\begin{equation}\label{def-F-counterexample}
\overline{F}(r)= \frac{4r- \sin (4r)}{\sin^2 (2r)}
\end{equation}
is a solution of \eqref{harmonicity-counterexample}. In other words, if $\overline{F}(r)$ is given by \eqref{def-F-counterexample}, a function as in \eqref{counterexample} is proper harmonic on $M\setminus \s^{q-1}$. By way of summary, it only remains to show that
\begin{equation}\label{counterexample-bis}
G(r,\vartheta_{2},\varphi_{2})=r\,F^*(r,\vartheta_{2},\varphi_{2})
\end{equation}
is not biharmonic. Again, by using twice \eqref{Laplacian-non-radiale-warped}, we find that $G(r,\vartheta_{2},\varphi_{2})$ is biharmonic if and only if $F(r)=r\,\overline{F(r)}$ is a solution of
\begin{equation}\label{biharmonicity-counterexample}
F^{(4)}(r)+8\,F^{(3)}(r)\cot (2r)+F''(r)\, \frac{8}{\sin^2 (2r)}\,\left(\cos (4r) -3\right )=0 \,\, .
\end{equation}
But direct substitution of $r\,\overline{F(r)}$ into the left-hand side of equation  \eqref{biharmonicity-counterexample} gives
$$
64 \cot(2r)\, \csc^4 (2r) \left ( \sin (4r)-4r \right) \not \equiv 0 \,\, ,
$$
so ending the proof.
In the case that $M$ is as in \eqref{join-hyperbolic}, one uses \eqref{Laplacian-radiale-warped} and \eqref{Laplacian-non-radiale-warped}: again, the calculations are entirely similar and so we omit the details.
\end{proof}
\begin{remark}
We conjecture that, under the assumptions of Theorem\1\ref{theorem-2-warped} (i), the function $G(r)=r^k\,F(r)$ is proper $(k+1)$-harmonic for all $k \geq1$.
\end{remark}
\section{The Almansi property on the Euclidean space and its generalization to the semi-Euclidean case}\label{proof-euclidean}
In the first part of the paper we showed that, in a Riemannian setting, the Almansi property is somehow peculiar of the Euclidean space. Here we provide an extension of this property to a semi-Euclidean context. Our proof complements the approach of \cite{ACL} and will include the classical Almansi property\1\ref{Almansi} as a special case.

Let $\R^{p,q}$ ($p\geq1,\,q \geq 0$) be equipped with its canonical $(p,q)$-signature metric structure
\begin{equation}\label{p-q-metric}
 g=\left[
 \begin{array}{cc}
  I_p &0\\
  0& -\,I_q
 \end{array}
 \right ] \,\,,
 \end{equation}
where $I_n$ is the identity matrix of order $n$. Then the Laplacian takes the form
 \begin{equation}\label{Laplacian-cartesian-p-q}
 \Delta F=\sum_{i=1}^p \, \frac{\partial^2 F}{\partial x_i^2}\,-\,\sum_{j=1}^q \, \frac{\partial^2 F}{\partial y_j^2}\,\,,
 \end{equation}
where $(x,y)=(x_1,\ldots,x_p,y_1,\ldots,y_q)$ are Cartesian coordinates. The case $q=0$ is admitted, so that the work of this section will include the Euclidean case as a special instance (note that \eqref{Laplacian} and Definitions\1\ref{iterated-laplacian}--\ref{definition-proper-r-harmonic} are extended to the semi-Euclidean case).

It is worth to point out that the general properties of solutions of the equation $\Delta F =0$ in a semi-Euclidean context are very much different from those of the classical harmonic functions on $\R^m$. By way of example, $F(x,y)=-\,(x^2+y^2)$ is a solution of $\Delta F =0$ on $\R^{1,1}$ but it does not satisfy the maximum principle. Despite these considerations, we can prove the following extension of the Almansi property\1\ref{Almansi} to the semi-Euclidean case:
\begin{property-general-semi}\label{almansi-semi} Let $\R^{p,q}$ be equipped with its canonical metric \eqref{p-q-metric} ($p\geq1,\,q \geq 0$). Let $H:\R^{p,q} \to \R$ be defined by
\begin{equation}\label{definizione-H-semi}
H(x_1,\ldots,x_p,y_1,\ldots,y_q)=c_1\,\left(\,\sum_{i=1}^p \, x_i^2 -\sum_{j=1}^q \, y_j^2\,\right)+c_2\,\, ,
\end{equation}
where $c_1,\,c_2$ are two real constants with $c_1 \neq 0$. If $F:U\subset \R^{p,q} \to \R$ is an $s$-harmonic function ($s \geq1$) on an open set $U$, then the product function $G=H \, F$ is $(s+1)$-harmonic on $U$. Moreover, if $F$ is proper $s$-harmonic on the whole $\R^{p,q}$, then $G$ is proper $(s+1)$-harmonic on $\R^{p,q}$.
\end{property-general-semi}
Because of the linearity, from now on we shall assume without further mention that $c_1=1$ and $c_2=0$ in \eqref{definizione-H-semi}.
Now, let us first carry out some preliminary work: with reference to the set of coordinates $w=(x,y)$ in \eqref{definizione-H-semi}, we introduce the following differential operators\begin{equation}\label{operatori-differenziali}
{\rm (i)} \,\, \left \{ \begin{array}{l}
\nabla^x =\left [\frac{\partial}{\partial x_1},\ldots, \frac{\partial}{\partial x_p}\right] \\
\\
\nabla^y =\left [\frac{\partial}{\partial y_1},\ldots, \frac{\partial}{\partial y_q}\right] \\
\end{array}
\right. \qquad {\rm (ii)} \,\, \left \{ \begin{array}{l}
\Delta^x =\sum_{i=1}^p\,\,\frac{\partial^2}{\partial x_i^2} \\
\\
\Delta^y =\sum_{j=1}^q\,\,\frac{\partial^2}{\partial y_j^2}\,\, .\\
\end{array}
\right.
\end{equation}
In particular, if $\nabla$ and $\Delta$ denote gradient and Laplacian respectively with respect to the metric structure \eqref{p-q-metric} on $\R^{p,q}$, we have:
\begin{equation}\label{legami-operatori-differenziali}
{\rm (i)} \,\, \nabla =[\nabla^x,\,-\nabla^y]\qquad {\rm (ii)} \,\,\Delta= \Delta^x - \Delta^y
\end{equation}
and, in particular,
\begin{equation}\label{gradiente-laplaciano-di-H}
{\rm (i)} \,\, \nabla H=2\, w\,\,  \qquad {\rm (ii)} \,\,\Delta H= 2  \, (p+q)\,\,.
\end{equation}
We now prove a series of lemmata of interest for our purposes.
\begin{lemma}\label{lemma1} Let $F$ be any smooth, real valued function defined on an open set $U$ of $\R^{p,q}$. Then
\begin{equation}\label{eq-1-almansi-semi}
\Delta \left( w \cdot \nabla F \right)= 2\, \Delta F + w \cdot \left (\nabla \left( \Delta F\right) \right ) \,\, .
\end{equation}
\end{lemma}
\begin{proof}We denote by $\cdot$ the scalar product on $\R^{p,q}$ and by $<,>$ the Euclidean scalar product of $\R^p$ and $\R^q$. First, we compute:
\begin{equation}\label{calcolo-utile}
w \cdot \nabla F= <x,\nabla^x F>-<y,-\,\nabla^y F>=\sum_{i=1}^p\,x_i\, \frac{\partial F}{\partial x_i}\,+\,\sum_{j=1}^q\,y_j\, \frac{\partial F}{\partial y_j}\,\,.
\end{equation}
From this we deduce:
\begin{eqnarray}
\Delta \left (w \cdot \nabla F \right )&=& \sum_{k=1}^p\, \frac{\partial^2}{\partial x_k^2}\left ( \sum_{i=1}^p\,x_i\, \frac{\partial F}{\partial x_i}\,\right )+ \sum_{k=1}^p\, \frac{\partial^2}{\partial x_k^2}\left ( \sum_{j=1}^q\,y_j\, \frac{\partial F}{\partial y_j} \right ) \\ \nonumber
&&-\, \sum_{\ell=1}^q\, \frac{\partial^2}{\partial y_\ell^2}\left ( \sum_{i=1}^p\,x_i\, \frac{\partial F}{\partial x_i}\,\right )- \sum_{\ell=1}^q\, \frac{\partial^2}{\partial y_\ell^2}\left ( \sum_{j=1}^q\,y_j\, \frac{\partial F}{\partial y_j} \right ) \\ \nonumber
&=&2\Delta^x F+<x,\nabla^x \left( \Delta^x F\right)>+ <y,\nabla^y \left( \Delta^x F\right)>\\ \nonumber
&&- 2\Delta^yF-<x,\nabla^x \left( \Delta^y F\right)>-<y,\nabla^y \left( \Delta^y F\right)>\\ \nonumber
&=&2\, \Delta F + w \cdot \left (\nabla \left( \Delta F\right) \right ) \,\, .\\ \nonumber
\end{eqnarray}
\end{proof}
By using Lemma\1\ref{lemma1} we obtain, by induction on $s$, the following:
\begin{lemma}\label{lemma2}Let $s \geq 1$ and let $F$ be any smooth, real valued function defined on an open set $U$ of $\R^{p,q}$. Then
\begin{equation}\label{eq-2-almansi-semi}
\Delta^s \left( w \cdot \nabla F \right)= 2\,s\, \Delta^s F + w \cdot \left (\nabla \left( \Delta^s F\right) \right ) \,\, .
\end{equation}
\end{lemma}

\begin{lemma}\label{lemma3}Let $s \geq 1$ and let $F$ be an $s$-harmonic function defined on an open set $U$ of $\R^{p,q}$. Then
\begin{equation}\label{eq-3-almansi-semi}
\Delta^s \left( H\, \Delta F  \right)= 0 \,\, .
\end{equation}
\end{lemma}
\begin{proof} By using \eqref{gradiente-laplaciano-di-H}, Lemma\1\ref{lemma2} and $\Delta^s F=0$ we compute:
\begin{eqnarray}\label{eq-lemma3}
\Delta^s \left(H\, \Delta F \right )&=&\Delta^{s-1} \left(\Delta \left (H\, \Delta F \right )\right)\\ \nonumber
&=&\Delta^{s-1} \left(2(p+q) \Delta F +H\, \Delta^2 F +4\,w \cdot \nabla \left(\Delta F \right)\right) \\ \nonumber
&=&2(p+q) \Delta^s F +\Delta^{s-1} \left(H\, \Delta^2 F \right) \\ \nonumber
&&\,+8\,(s-1)\,\Delta^s F + 4\,w \cdot \nabla \left(\Delta^s F \right) \\ \nonumber
&=&\Delta^{s-1} \left(H\, \Delta^2 F \right)= \ldots= \Delta \left (H\,\Delta^s F\right )=0\,\,.\\ \nonumber
\end{eqnarray}
\end{proof}
\begin{lemma}\label{lemma4} Let $c_1,\,c_2$ be two positive constants. If $F^*:\R^n \to \R$ ($n\geq1$) is a smooth solution of
\begin{equation}\label{eq-4-almansi-semi}
c_1 \,F^* + c_2\, x \cdot \nabla F^* =0\,\,,
\end{equation}
then $F \equiv 0$.
\end{lemma}
\begin{proof} A proof could be derived by the arguments of Lemma\11.2 of \cite{ACL} (with $c=0$ there). Since it is not convenient to introduce the notation of \cite{ACL}, we provide here an alternative, simple proof. First, we observe that the validity of the equation \eqref{eq-4-almansi-semi} implies $F^*(O)=0$. It follows that, if $F^*$ is constant, then $F \equiv 0$. Next, we assume that $F^*$ is not constant and derive a contradiction. Let
\begin{equation}\label{Omega+}
\Omega^+= \left \{x \in \R^n \,\, : \,\, F^*(x)>0 \right \}\,\,.
\end{equation}
By working with $-\,F^*$ if necessary, we can assume that $\Omega^+$ is not empty and that $O \in \overline{\Omega^+}$. It follows from \eqref{eq-4-almansi-semi} that $x \cdot \nabla F^*<0$ on $\Omega^+$. That means that, inside $\Omega^+$, the gradient $\nabla F^*$ always  points towards the interior of any ball $B_R(O)$. This contradicts the fact that in $\Omega^+$ we must have a path along which $F^*$ decreases to $0$.
\end{proof}

Now, we are in the right position to complete the proof of the generalized Almansi property\1\ref{almansi-semi}:
\begin{proof}
We assume $\Delta^s F=0$ on an open set $U$, we use \eqref{Laplaciano-prodotto}, \eqref{gradiente-laplaciano-di-H}, Lemmata\1\ref{lemma2}, \ref{lemma3} and compute:
\begin{eqnarray}\label{Proof-Almansi-2}
\Delta^{s+1} (H\,F)&=& \Delta^s \left( \Delta (H\,F) \right) \\ \nonumber
&=&\Delta^s \left(2(p+q)\, F \right )+\Delta^s \left(H\, \Delta F \right )+4 \,\Delta^s \left( w \cdot \nabla F \right)\\ \nonumber
&=& 0\,\,.\\ \nonumber
\end{eqnarray}
Now, it remains to show that, if $F$ if a globally defined proper $s$-harmonic function, then $H\,F$ is \textit{proper} $(s+1)$-harmonic on $\R^{p,q}$. By computing precisely as in the proof of Lemma\1\ref{lemma3}, we find that
$$
\Delta^{s}\left( H\,F \right)= c_1\, \Delta^{s-1}F+c_2\,w \cdot \nabla \left( \Delta^{s-1}F \right)
$$
for some positive constants $c_1,\,c_2$ which depend on $s,p$ and $q$. Next, we observe that, by hypothesis, the function $F^*=\Delta^{s-1}F$ does not vanish identically because $F$ is a proper $s$-harmonic function. We consider $F^*$ as a function which is globally defined on $\R^n,\,n=p+q$. As a consequence of \eqref{calcolo-utile}, the conclusion follows immediately from Lemma\1\ref{lemma4}.
\end{proof}
\begin{remark} As an immediate consequence of the generalized Almansi property\1\ref{almansi-semi}, we observe that if $F:\R^{p,q} \to \R$ is a proper harmonic function, then $H^s\, F$, with $H$ as in \eqref{definizione-H-semi}, is proper $(s+1)$-harmonic for all $s \geq 1$.
\end{remark}
\begin{remark} All the calculations and lemmata involved in the proof of the Almansi property on $\R^{p,q}$ are of a local nature, with the exception of Lemma\1\ref{lemma4} where it is required that $F^*$ be globally defined. In particular, we point out that if $F$ is a proper $s$-harmonic function defined on an open set $U$ of $\R^{p,q}$, then $G=H \,F$ is $(s+1)$-harmonic but, in general, it may not be proper $(s+1)$-harmonic. To illustrate this type of situation, it is enough to consider the following simple example in $\R^2$:
$$
F(x,y)= \frac{x}{x^2+y^2}
$$
is proper harmonic on $\R^2 \setminus\{O\}$, and $G=H \,F$ is clearly biharmonic but not proper biharmonic. Similarly, in $\R^{1,1}\setminus\{x=\pm y\}$,
$$
F(x,y)= \frac{x}{x^2-y^2}
$$
is proper harmonic, and $G=H \,F$ is biharmonic but not proper biharmonic.
\end{remark}

%%%%%%%%%%%%%%%%%%%%
%%%%%%%%%%%%%%%%%%%%%%%
%%%%%%%%%%%%%%%%%%%%%%%%%%%
%We shall prove property\1\ref{almansi-semi} by induction on $s$. So, let us first assume that $s=1$ and that $F$ is a proper harmonic function. By using $\Delta F=0$, \eqref{Laplaciano-prodotto}, \eqref{gradiente-laplaciano-di-H} and Lemma\1\ref{lemma1} we have:
%\begin{eqnarray}\label{Proof-Almansi-1}
%\Delta^2 (H\,F)&=& \Delta \left( F\,\Delta H  + 2 \, \nabla H \cdot \nabla F \right) \\ \nonumber
%&=&=\Delta \left(2\,(p+q)\,F+4\,w \cdot \nabla F \right )\\ \nonumber
%&=&4\,\Delta \left( w \cdot \nabla F \right)\\ \nonumber
%&=& 8\, \Delta F + 4\,w \cdot \left (\nabla \left( \Delta F\right) \right ) =0 \,\,.
%\end{eqnarray}
%Since $\Delta (H\,F)=2\,(p+q)\,F+4\,w \cdot \nabla F\not \equiv  0$ because $F$ is proper harmonic (andrebbe spiegato meglio!!), in the case that $s=1$ the proof of property\1\ref{almansi-semi} is completed. Now we suppose that property\1\ref{almansi-semi} is true for $s$ and we deduce that it holds for $(s+1)$.
%%%%%%%%%%%%%%%%%%%%
%%%%%%%%%%%%%%%%%%%%%%%
%%%%%%%%%%%%%%%%%%%%%%%%%%%

\begin{thebibliography}{99}
\bibitem{Almansi} E.~Almansi. Sull'integrazione dell'equazione differenziale $\Delta^n=0$. {\it Annali di Matematica}, Novembre (1898), 1--51.

\bibitem{ACL} N.~Aronszajn, T.M.~Creese, L.J.~Lipkin. Polyharmonic functions. {\it Oxford Science Publications}, (1983).

\bibitem{BGM} M.~Berger, P.~Gauduchon, E.~Mazet. Le spectre d'une vari\'et\'e riemannienne. {\it Springer Lecture Notes}, 194 (1971).

\bibitem{Caddeo1} R.~Caddeo. Riemannian manifolds on which the distance function is biharmonic. {\it Rend. Sem. Mat. Univ. e Politec. Torino}, 40 (1982), 93--101.

\bibitem{Caddeo2} R.~Caddeo. Sulla propriet\`a di Almansi della funzione distanza su una variet\`a Riemanniana. {\it Conf. Sem. Mat. Univ. Bari}, 198 (1984), 1--25.


\bibitem{cecil} T.E.~Cecil,  P.J.~Ryan. Geometry of hypersurfaces. {\it Springer Monographs in Mathematics}. Springer, New York, (2015).
%\bibitem{Cartan} E.~Cartan. Familles de surfaces isoparam\'etriques dans les espaces \`a courbure constante. {\it Ann. Mat. Pura Appl.}, 17 (1938), 177--191.

\bibitem{dassios} G.~Dassios. Ellipsoidal harmonics. Theory and applications. {\it Encyclopedia of Mathematics and its Applications}, 146. Cambridge University Press, Cambridge, (2012).

\bibitem{dassios2} G.~Dassios. The Kelvin transformation in potential theory and Stokes flow. {\it IMA J. Appl. Math.}, 74 (2009), 427--438.

\bibitem{GW} R.E.~Greene, H.~Wu.
Function theory on manifolds which possess a pole. {\it Lecture Notes in Mathematics} 699. Springer, Berlin, (1979).

\bibitem{GMR} S.~Gudmundsson, S.~Montaldo, A.~Ratto. Biharmonic functions on the classical compact simple Lie groups. {\it J. Geom. Anal.}, to appear.

\bibitem{MO} S.~Montaldo, C.~Oniciuc. {A short survey on biharmonic maps between Riemannian manifolds}. {\it Rev. Un. Mat. Argentina}, 47 (2006), 1--22.

\bibitem{MOR} S.~Montaldo, C.~Oniciuc, A.~Ratto. {Rotationally symmetric biharmonic maps between models}. {\em J. Math. Anal. Appl.}, 431 (2015), 494-508.

\bibitem{MR} S.~Montaldo, A.~Ratto. A General approach to equivariant biharmonic Maps. {\em Med. J. Math.}, 10 (2013), 1127--1139.

%\bibitem{M} H.F.~M\"{u}nzner. Isoparametrische Hyperfl\"{a}chen in Sph\"{a}ren. {\em Math. Ann.}, 256 (1981), 215--232.


\bibitem{Petersen} P.~Petersen. Riemannian geometry. {\em Graduate texts in Mathematics, Springer}, 171 (2006).

\bibitem{thorbergsson}  G.~Thorbergsson. A survey on isoparametric hypersurfaces and their generalizations. {\em Handbook of differential geometry}, Vol. I, 963--995, North-Holland, Amsterdam, (2000).


\bibitem{VanHecke} A.~Gray, L.~Vanhecke. Riemannian geometry as determined by the volume of small geodesic balls. {\it Acta Math.} 142 (1979) 157--198.

\end{thebibliography}
\end{document}